\documentclass[english,11pt,reqno]{smfart}

\usepackage{color}
\usepackage{amsthm}
\usepackage{amsmath}
\usepackage{amsfonts}
\usepackage{amstext}
\usepackage[latin1]{inputenc}
\usepackage{amscd}
\usepackage{latexsym}
\usepackage{bm}
\usepackage{amssymb}
\usepackage[all]{xy}
\usepackage{euscript}
\usepackage{a4wide}
\usepackage{amsmath,amssymb,graphicx}
\usepackage{amssymb}
\usepackage{amsmath}
\usepackage{mathrsfs,mathtools}

\newtheorem{theorem}{Theorem}
\newtheorem{proposition}{Proposition}
\newtheorem{property}{Property}
\newtheorem{lemma}{Lemma}

\newtheorem{definition}{Definition}

\def\di{\displaystyle}

\newcommand{\N}{\mathbb{N}}
\newcommand{\R}{\mathbb{R}}

\newcommand{\LL}{\mathcal{L}}
\newcommand{\CC}{\mathscr{C}}

\newcommand{\DM}{D^\alpha_-}
\newcommand{\DP}{D^\alpha_+}

\newcommand{\E}{\mathrm{E}_{\alpha,p}}
\renewcommand{\L}{\mathrm{L}}
\newcommand{\W}{\mathrm{W}}

\newcommand*{\hooktwoheadrightarrow}{\lhook\joinrel\twoheadrightarrow}
\newcommand{\fonction}[5]{\begin{array}[t]{lrcl}#1 :&#2 &\longrightarrow &#3\\&#4& \longmapsto &#5 \end{array}}

\begin{document}
\title{Existence of a weak solution for fractional Euler-Lagrange equations.}
\author{Lo\"ic Bourdin}
\address{Laboratoire de Math\'ematiques et de leurs Applications - Pau (LMAP). UMR CNRS 5142. Universit\'e de Pau et des Pays de l'Adour.}
\email{bourdin.l@etud.univ-pau.fr}
\maketitle.

\begin{abstract}
We derive sufficient conditions ensuring the existence of a weak solution $u$ for fractional Euler-Lagrange equations of the type:
\begin{equation}\tag{EL${}^\alpha $}
\dfrac{\partial L}{\partial x} (u,\DM u,t) + \DP \left( \dfrac{\partial L}{\partial y} (u,\DM u,t) \right) = 0,
\end{equation}
on a real interval $[a,b]$ and where $\DM$ and $\DP$ are the fractional derivatives of Riemann-Liouville of order $0 < \alpha < 1$.
\end{abstract}



\section{Introduction}\label{section1}
\subsection{Context in the fractional calculus}\label{section11}
The mathematical field that deals with derivatives of any real order is called fractional calculus. For a long time, it was only considered as a pure mathematical branch. Nevertheless, during the last two decades, fractional calculus has attracted the attention of many researchers and it has been successfully applied in various areas like computational biology \cite{magi} or economy \cite{comt}. In particular, the first and well-established application of fractional operators was in the physical context of anomalous diffusion, see \cite{neel,neel2} for example. Let us mention \cite{metz} proving that fractional equations is a complementary tool in the description of anomalous transport processes. We refer to \cite{hilf3} for a general review of the applications of fractional calculus in several fields of Physics. In a more general point of view, fractional differential equations are even considered as an alternative model to non-linear differential equations, see \cite{boni}. \\

For the origin of the calculus of variations with fractional operators, we should look back to 1996-97 when Riewe used non-integer order derivatives to better describe non conservative systems in mechanics \cite{riew,riew2}. Since then, numerous works on the fractional variational calculus have been made. For instance, in the same spirit, authors of \cite{cres8,cres7} have recently derived fractional variational structures for non conservative equations. Furthermore, one can find a comprehensive literature regarding necessary optimality conditions and Noether's theorem, see \cite{agra,alme,bale2,bour2,torr3,odzi4}. Concerning the state of the art on the fractional calculus of variations and respective fractional Euler-Lagrange equations, we refer the reader to the recent book \cite{torr5}. \\

In the whole paper, we consider $a < b$ two reals, $d \in \N^*$ and the following Lagrangian functional
\begin{equation}
\LL (u) = \di \int_a^b L(u,\DM u,t) \;dt,
\end{equation}
where $L$ is a Lagrangian, \textit{i.e.} a map of the form:
\begin{equation}
\fonction{L}{\R^d \times \R^d \times [a,b]}{\R}{(x,y,t)}{L(x,y,t),}
\end{equation}
where $\DM$ is the left fractional derivative of Riemann-Liouville of order $0 < \alpha < 1$ and where the variable $u$ is a function defined almost everywhere (shortly a.e.) on $(a,b)$ with values in $\R^d$. The precise definitions of the fractional operators of Riemann-Liouville will be recalled in Section~\ref{section22}. It is well-known that critical points of the functional $\LL$ are characterized by the solutions of the \textit{fractional Euler-Lagrange equation}: 
\begin{equation}\tag{EL${}^\alpha $}\label{elf}
\dfrac{\partial L}{\partial x} (u,\DM u,t) + \DP \left( \dfrac{\partial L}{\partial y} (u,\DM u,t) \right) = 0,
\end{equation}
where $\DP$ is the right fractional derivative of Riemann-Liouville, see detailed proofs in \cite{agra,bale2} for example. However, as far as the author is aware and despite particular results in \cite{jiao,klim}, no existence result of a solution for \eqref{elf} exists in a general case. \\

The aim of this paper is to derive \textbf{sufficient conditions on $L$ so that \eqref{elf} admits a weak solution}. \\ 

Let us note that, in a more general setting, existence results for fractional equations is an emerging field. For instance, there are recent results about existence and uniqueness of solution for a class of fractional evolution equations in \cite{zhou,zhou2}.

\subsection{Main result}\label{section12}
We denote by $\Vert \cdot \Vert$ the Euclidean norm of $\R^d$ and $\CC := \CC ([a,b];\R^d)$ the space of continuous functions endowed with its usual norm $\Vert \cdot \Vert_{\infty} $. 

\begin{definition}
A function $u$ is said to be a weak solution of \eqref{elf} if $u \in \CC$ and if $u$ satisfies \eqref{elf} a.e. on $[a,b]$.
\end{definition}

Let us enunciate the main result of the paper:
\begin{theorem}\label{thmprincipal}
Let $L$ be a Lagrangian of class $\CC^1$ and $0 < (1/p) < \alpha < 1$. If $L$ satisfies the following hypotheses denoted by \eqref{h1}, \eqref{h2}, \eqref{h3}, \eqref{h4} and \eqref{h5}:
\begin{itemize}
\item there exist $0 \leq d_1 \leq p$ and $r_1$, $s_1 \in \CC(\R^d \times [a,b], \R^+)$ such that:
\begin{equation}\tag{H${}_1$}\label{h1}
\forall (x,y,t) \in \R^d \times \R^d \times [a,b], \; \vert L(x,y,t) - L(x,0,t) \vert \leq r_1 (x,t) \Vert y \Vert^{d_1}+s_1 (x,t);
\end{equation}
\item there exist $0 \leq d_2 \leq p$ and $r_2$, $s_2 \in \CC(\R^d \times [a,b], \R^+)$ such that:
\begin{equation}\tag{H${}_2$}\label{h2}
\forall (x,y,t) \in \R^d \times \R^d \times [a,b], \; \left\Vert \dfrac{\partial L}{\partial x} (x,y,t) \right\Vert \leq r_2 (x,t) \Vert y \Vert^{d_2}+s_2 (x,t);
\end{equation}
\item there exist $0 \leq d_3 \leq p-1$ and $r_3$, $s_3 \in \CC(\R^d \times [a,b], \R^+)$ such that:
\begin{equation}\tag{H${}_3$}\label{h3}
\forall (x,v,t) \in \R^d \times \R^d \times [a,b], \; \left\Vert \dfrac{\partial L}{\partial y} (x,y,t) \right\Vert \leq r_3 (x,t) \Vert y \Vert^{d_3}+s_3 (x,t);
\end{equation}
\item \textbf{coercivity condition}: there exist $\gamma > 0$, $1 \leq d_4 < p$, $c_1 \in \CC(\R^d \times [a,b], [\gamma,\infty[)$, $c_2$, $c_3 \in \CC([a,b], \R)$ such that:
\begin{equation}\tag{H${}_4$}\label{h4}
\forall (x,y,t) \in \R^d \times \R^d \times [a,b], \; L(x,y,t) \geq c_1 (x,t) \Vert y \Vert^p + c_2 (t) \Vert x \Vert^{d_4}+c_3(t);
\end{equation}
\item \textbf{convexity condition}:
\begin{equation}\tag{H${}_5$}\label{h5}
\forall t \in [a,b], \; L(\cdot,\cdot,t) \; \text{is convex},
\end{equation}
\end{itemize}
then \eqref{elf} admits a weak solution.
\end{theorem}

Hypotheses denoted by \eqref{h1}, \eqref{h2}, \eqref{h3} are usually called \textbf{regularity hypotheses}, see \cite{cesa,daco2}. In Section~\ref{section5}, we prove that Hypothesis \eqref{h5} can be replaced by different convexity assumptions. 

\subsection{Idea of the proof of Theorem~\ref{thmprincipal}}\label{section13}
In the classical case $\alpha = 1$, $D^1_- = - D^1_+ = d/dt $ and consequently \eqref{elf} is nothing else but the classical Euler-Lagrange equation formulated in the 1750's. In this case, a lot of results of existence of solutions have been already proved. Let us recall that there exist different approaches:
\begin{itemize}
\item A first approach is to develop the classical Euler-Lagrange equation in order to obtain an implicit second order differential equation, see \cite{godb}. Then, under a \textit{hyper regularity} or \textit{non singularity} condition on the Lagrangian $L$, the equation can be written as an explicit second order differential equation and the Cauchy-Lipschitz theorem gives the existence of local or global regular solutions;
\item A second approach consists in using the variational structure of the equation, see \cite{daco2}. Indeed, under some assumptions, the critical points of $\LL$ correspond to the solutions of the classical Euler-Lagrange equation. The idea is then to prove the existence of critical points of $\LL$. In this way, author makes some assumptions (like coercivity and convexity of the Lagrangian $L$) ensuring the existence of extrema of $\LL$. With this second method, author has to use reflexive spaces of functions and consequently, he deals with weak solutions (in a specific sense).
\end{itemize}
In order to prove Theorem~\ref{thmprincipal}, we extend the second approach to the \textbf{strict fractional case} (\textit{i.e.} $0<\alpha <1$). Indeed, although there exist fractional versions of the Cauchy-Lipschitz theorem (see \cite{kilb,samk}), there is no simple rules for the fractional derivative of a composition and consequently, we can not write \eqref{elf} in a simpler way. Hence, in the strict fractional case, we can not follow the first method. \\

Theorem~\ref{thmprincipal} is based on the following preliminaries:
\begin{itemize}
\item The introduction in Section~\ref{section3} of an appropriate reflexive separable Banach space $\E$ (see \eqref{eqdefE});
\item Assuming Hypotheses \eqref{h1}, \eqref{h2} and \eqref{h3}, Theorem~\ref{thm1} in Section~\ref{section4} states that if $u$ is a critical point of $\LL$, then $u$ is a weak solution of \eqref{elf};
\item Assuming additionally Hypotheses \eqref{h4} and \eqref{h5}, Theorem~\ref{thm2} in Section~\ref{section5} states that $\LL$ admits a global minimizer.
\end{itemize}
Hence, the proof of Theorem~\ref{thmprincipal} is complete. Let us note that the method developed in this paper is inspired by:
\begin{itemize}
\item the reflexive separable Banach space introduced in \cite{jiao} allowing to prove the existence of a weak solution for a class of fractional boundary value problems;
\item the suitable hypotheses of regularity, coercivity and convexity given in \cite{daco2} proving the existence of a weak solution for classical Euler-Lagrange equations (\textit{i.e.} in the case $\alpha =1$).
\end{itemize}

\subsection{Organisation of the paper}\label{section14}
The paper is organized as follows. In Section~\ref{section2}, some usual notations of spaces of functions are given. We recall the definitions of the fractional operators of Riemann-Liouville and some of their properties. Section~\ref{section3} is devoted to the introduction and to the study of the appropriate reflexive separable Banach space $\E$. In Section~\ref{section4}, the variational structure of \eqref{elf} is considered and we prove Theorem~\ref{thm1}. In Section~\ref{section5}, we prove Theorem~\ref{thm2}. Then, Section~\ref{section6} is devoted to some examples. Finally, a conclusion ends this paper.

\section{Reminder about fractional calculus}\label{section2}
\subsection{Some spaces of functions}\label{section21}
For any $ p \geq 1$, $\L^p := \L^p \big( (a,b);\R^d \big)$ denotes the classical Lebesgue space of $p$-integrable functions endowed with its usual norm $\Vert \cdot \Vert_{\L^p} $. Let us give some usual notations of spaces of continuous functions defined on $[a,b]$ with values in $\R^d$:
\begin{itemize}
\item $AC := AC ([a,b];\R^d)$ the space of absolutely continuous functions;
\item $\CC^\infty := \CC^\infty ([a,b];\R^d)$ the space of infinitely differentiable functions;
\item $\CC^\infty_c := \CC^\infty_c ([a,b];\R^d)$ the space of infinitely differentiable functions and compactly supported in $]a,b[$.
\end{itemize}
We remind that a function $f$ is an element of $AC$ if and only if $\dot{f} \in \L^1$ and the following equality holds:
\begin{equation}
\forall t \in [a,b], \; f(t)=f(a) + \di \int_a^t \dot{f}(\xi) \; d\xi ,
\end{equation}
where $\dot{f}$ denotes the derivative of $f$. We refer to \cite{kolm} for more details concerning the absolutely continuous functions. \\

Finally, we denote by $\CC_{a}$ (resp. $AC_a$ or $\CC^\infty_a$) the space of functions $f \in \CC$ (resp. $AC$ or $\CC^\infty$) such that $f(a) = 0$. In particular, $ \CC^\infty_c \subset \CC^\infty_a \subset AC_a$. \\

\textbf{Convention:} in the whole paper, an equality between functions must be understood as an equality holding for almost all $t \in (a,b)$. When it is not the case, the interval on which the equality is valid will be specified.

\subsection{Fractional operators of Riemann-Liouville}\label{section22}
Since 1695, numerous notions of fractional operators emerged over the year, see \cite{kilb,podl,samk}. In this paper, we only deal with the fractional operators of Riemann-Liouville (1847) whose definitions and some basic results are reminded in this section. We refer to \cite{kilb,samk} for the omitted proofs. \\
 
Let $\alpha > 0$ and $f$ be a function defined a.e. on $(a,b)$ with values in $\R^d$. The left (resp. right) fractional integral in the sense of Riemann-Liouville with inferior limit $a$ (resp. superior limit $b$) of order $ \alpha $ of $f$ is given by:
\begin{equation}
\forall t \in ]a,b], \; I^{\alpha}_- f (t) := \dfrac{1}{\Gamma (\alpha)} \di \int_a^t (t-\xi)^{\alpha -1} f(\xi) \; d\xi,
\end{equation}
respectively:
\begin{equation}
\forall t \in [a,b[, \; I^{\alpha}_+ f (t) := \dfrac{1}{\Gamma (\alpha)} \di \int_t^b (\xi-t)^{\alpha -1} f(\xi) \; d\xi,
\end{equation}
where $\Gamma$ denotes the Euler's Gamma function. If $f \in \L^1$, then $I^{\alpha}_- f$ and $I^{\alpha}_+ f$ are defined a.e. on $(a,b)$. \\

Now, let us consider $ 0 < \alpha < 1$. The left (resp. right) fractional derivative in the sense of Riemann-Liouville with inferior limit $a$ (resp. superior limit $b$) of order $\alpha$ of $f$ is given by:
\begin{equation}\label{eq211}
\forall t \in ]a,b], \; \DM f(t) := \dfrac{d}{dt} \big( I^{1-\alpha}_- f \big) (t) \quad \Big( \text{resp.} \quad \forall t \in [a,b[, \; \DP f(t) := -\dfrac{d}{dt} \big( I^{1-\alpha}_+ f \big) (t) \Big).
\end{equation}
From \cite[Corollary 2.2, p.73]{kilb}, if $f \in AC$, then $\DM f$ and $\DP f$ are defined a.e. on $(a,b)$ and satisfy:
\begin{equation}\label{eq21}
\DM f = I^{1-\alpha}_- \dot{f} + \dfrac{f(a)}{(t-a)^\alpha \Gamma (1-\alpha)} \quad \text{and} \quad \DP f = - I^{1-\alpha}_+ \dot{f} + \dfrac{f(b)}{(b-t)^\alpha \Gamma (1-\alpha)}.
\end{equation}
In particular, if $f \in AC_a$, then $\DM f = I^{1-\alpha}_- \dot{f}$.

\subsection{Some properties of the fractional operators}\label{section23} 
In this section, we provide some properties concerning the left fractional operators of Riemann-Liouville. One can easily derive the analogous versions for the right ones. Properties~\ref{property1}, \ref{property2} and \ref{property3} are well-known and one can find their proofs in the classical literature on the subject (see \cite[Lemma 2.3, p.73]{kilb}, \cite[Lemma 2.1, p.72]{kilb} and \cite[Lemma 2.7, p.76]{kilb} respectively). \\

The first result yields the semi-group property of the left Riemann-Liouville fractional integral:
\begin{property}\label{property1}
For any $\alpha$, $\beta > 0$ and any function $f \in \L^1$, the following equality holds:
\begin{equation}
I^\alpha_- \circ I^\beta_- f = I^{\alpha+\beta}_- f .
\end{equation}
\end{property}
From Property~\ref{property1} and Equalities \eqref{eq211} and \eqref{eq21}, one can easily deduce the following results concerning the composition between fractional integral and fractional derivative. For any $0 < \alpha <1$, the following equalities hold:
\begin{equation}\label{eq24}
\forall f \in \L^1, \; D^\alpha_- \circ I^\alpha_- f = f \quad \text{and} \quad \forall f \in AC, \; I^\alpha_- \circ \DM  f= f.
\end{equation}

Another classical result is the boundedness of the left fractional integral from $\L^p$ to $\L^p$:
\begin{property}\label{property2}
For any $\alpha >0$ and any $p \geq 1$, $I^\alpha_-$ is linear and continuous from $\L^p$ to $\L^p$. Precisely, the following inequality holds:
\begin{equation}
\forall f \in \L^p, \; \Vert I^\alpha_- f \Vert_{\L^p} \leq \dfrac{(b-a)^\alpha}{\Gamma (1+\alpha)} \Vert f \Vert_{\L^p}.
\end{equation} 
\end{property}

The following classical property concerns the integration of fractional integrals. It is occasionally called \textit{fractional integration by parts}: 

\begin{property}\label{property3}
Let $0 < \alpha < 1$. Let $f \in \L^p$ and $g \in \L^q$ where $(1/p)+(1/q) \leq 1 + \alpha$ (and $p \neq 1 \neq q$ in the case $(1/p)+(1/q) = 1 + \alpha$). Then, the following equality holds:
\begin{equation}
\di \int_a^b I^\alpha_- f \cdot g \; dt = \di \int_a^b f \cdot I^\alpha_+ g \; dt.
\end{equation}
\end{property}

This change of side of the fractional integral (from $I^\alpha_- $ to $I^\alpha_+ $) is responsible of the emergence of $D^\alpha_+ $ in \eqref{elf} although only $D^\alpha_- $ is involved in the Lagrangian functional $\LL$. We refer to Section~\ref{section42} for more details. \\

The following Property~\ref{property4} completes Property~\ref{property2} in the case $0 < (1/p) < \alpha < 1$: indeed, in this case, $I^\alpha_- $ is additionally bounded from $\L^p$ to $\CC_a$:

\begin{property}\label{property4}
Let  $0 < (1/p) < \alpha < 1$ and $q=p/(p-1)$. Then, for any $f \in \L^p$, we have:
\begin{itemize}
\item $I^\alpha_- f$ is Hold\"er continuous on $]a,b]$ with exponent $ \alpha - (1/p) > 0$;
\item $\lim\limits_{t \to a} I^\alpha_- f (t) = 0$.
\end{itemize}
Consequently, $I^\alpha_- f$ can be continuously extended by $0$ in $t=a$. Finally, for any $f \in \L^p$, we have $I^\alpha_- f \in \CC_a$. Moreover, the following inequality holds:
\begin{equation}
\forall f \in \L^p, \; \Vert I^\alpha_- f \Vert_{\infty} \leq  \dfrac{(b-a)^{\alpha -(1/p)}}{\Gamma (\alpha) \big( (\alpha -1) q + 1\big)^{1/q} } \Vert f \Vert_{\L^p} .
\end{equation}
\end{property}

\begin{proof}
Let us note that this result is mainly proved in \cite{jiao}. Let $f \in \L^p$. We first remind the following inequality:
\begin{equation}
\forall \xi_1 \geq \xi_2 \geq 0, \; (\xi_1-\xi_2)^q \leq \xi_1^q - \xi_2^q.
\end{equation}
Let us prove that $I^\alpha_- f$ is Hold\"er continuous on $]a,b]$. For any $a < t_1 < t_2 \leq b$, using the H\"older's inequality, we have:
\begin{eqnarray*}
\Vert I^\alpha_- f(t_2) - I^\alpha_- f(t_1) \Vert & = & \dfrac{1}{\Gamma (\alpha)} \left\Vert \di \int_a^{t_2} (t_2 - \xi )^{\alpha -1} f(\xi) \; d\xi - \int_a^{t_1} (t_1 - \xi )^{\alpha -1} f(\xi) \; d\xi \right\Vert \\
& \leq & \dfrac{1}{\Gamma (\alpha)} \left\Vert \di \int_{t_1}^{t_2} (t_2 - \xi )^{\alpha -1} f(\xi) \; d\xi \right\Vert \\ 
& & \qquad \qquad \qquad + \dfrac{1}{\Gamma (\alpha)} \left\Vert \di \int_a^{t_1} \big( (t_2 - \xi )^{\alpha -1} - (t_1 - \xi )^{\alpha -1} \big) f(\xi) \; d\xi \right\Vert   \\
& \leq & \dfrac{\Vert f \Vert_{\L^p}}{\Gamma (\alpha)} \left( \di \int_{t_1}^{t_2} (t_2 -\xi)^{(\alpha -1)q} \; d\xi \right)^{1/q} \\ 
& & \qquad \qquad \qquad + \dfrac{\Vert f \Vert_{\L^p}}{\Gamma (\alpha)} \left( \di \int_{a}^{t_1} \big( (t_1 -\xi)^{\alpha -1} - (t_2 -\xi)^{\alpha -1} \big)^q \; d\xi \right)^{1/q} \\
& \leq & \dfrac{\Vert f \Vert_{\L^p}}{\Gamma (\alpha)} \left( \di \int_{t_1}^{t_2} (t_2 -\xi)^{(\alpha -1)q} \; d\xi \right)^{1/q} \\ 
& & \qquad \qquad \qquad + \dfrac{\Vert f \Vert_{\L^p}}{\Gamma (\alpha)} \left( \di \int_a^{t_1} (t_1 -\xi) ^{(\alpha -1)q} - (t_2 -\xi) ^{(\alpha -1)q} \; d\xi \right) ^{1/q} \\
& \leq & \dfrac{2 \Vert f \Vert_{\L^p}}{\Gamma (\alpha) \big( (\alpha -1) q + 1\big)^{1/q} } (t_2 - t_1)^{\alpha -(1/p)}.
\end{eqnarray*}
The proof of the first point is complete. Let us consider the second point. For any $t \in ]a,b]$, we can prove in the same manner that:
\begin{equation}
\Vert I^\alpha_- f(t) \Vert \leq \dfrac{\Vert f \Vert_{\L^p}}{\Gamma (\alpha) \big( (\alpha -1) q + 1\big)^{1/q} } (t-a)^{\alpha -(1/p)} \xrightarrow[t \to a]{} 0.
\end{equation}
The proof is now complete.
\end{proof}

\section{Space of functions $\E$}\label{section3}
In order to prove the existence of a weak solution of \eqref{elf} using a variational method, we need the introduction of an appropriate space of functions. This space has to present some properties like reflexivity, see \cite{daco2}. \\

For any $0 < \alpha < 1$ and any $ p \geq 1$, we define the following space of functions:
\begin{equation}\label{eqdefE}
\E := \{ u \in \L^p \; \text{satisfying} \; \DM u \in \L^p \; \text{and} \; I^\alpha_- \circ D^\alpha_- u = u \; \; \text{a.e.} \}.
\end{equation}
We endow $\E$ with the following norm:  
\begin{equation}
\fonction{\Vert \cdot \Vert_{\alpha,p}}{\E}{\R^+}{u}{\big( \Vert u \Vert_{\L^p}^p  + \Vert \DM u \Vert_{\L^p}^p \big)^{1/p}.}
\end{equation}
Let us note that:
\begin{equation}
\fonction{\vert \cdot \vert_{\alpha,p}}{\E}{\R^+}{u}{\Vert D^\alpha_- u \Vert_{\L^p}} 
\end{equation}
is an equivalent norm to $\Vert \cdot \Vert_{\alpha,p}$ for $\E$. Indeed, Property~\ref{property2} leads to:
\begin{equation}\label{eq54}
\forall u \in \E, \; \Vert u \Vert_{\L^p} = \Vert I^\alpha_- \circ D^\alpha_- u \Vert_{\L^p} \leq \dfrac{(b-a)^\alpha}{\Gamma (1+\alpha)} \Vert \DM u \Vert_{\L^p}.
\end{equation}
The goal of this section is to prove the following proposition:
\begin{proposition}\label{prop556}
Assuming $0 < (1/p) < \alpha < 1$, $\E$ is a reflexive separable Banach space and the compact embedding $\E \hooktwoheadrightarrow \CC_a$ holds. 
\end{proposition}

Then, in the rest of the paper, we consider:
\begin{equation}
0 < (1/p) < \alpha < 1 \quad \text{and} \quad q=p/(p-1).
\end{equation}
Let us detail the different points of Proposition~\ref{prop556} in the following subsections.

\subsection{$\E$ is a reflexive separable Banach space}\label{section32}
Let us prove this property. Let us consider $(\L^p)^2$ the set $\L^p \times \L^p$ endowed with the norm $\Vert (u,v) \Vert_{(\L^p)^2} = (\Vert u \Vert_{\L^p}^p + \Vert v \Vert_{\L^p}^p )^{1/p}$. Since $p >1$, $(\L^p, \Vert \cdot \Vert_{\L^p})$ is a reflexive separable Banach space and therefore, $\big((\L^p)^2, \Vert \cdot \Vert_{(\L^p)^2}\big)$ is also a reflexive separable Banach space. \\

We define $\Omega := \{ (u,\DM u ), \; u \in \E \}$. Let us prove that $\Omega$ is a closed subspace of $\big((\L^p)^2, \Vert \cdot \Vert_{(\L^p)^2}\big)$. Let $(u_n,v_n)_{n \in \N} \subset \Omega$ such that:
\begin{equation}
(u_n,v_n) \xrightarrow[]{(\L^p)^2} (u,v).
\end{equation}
Let us prove that $(u,v) \in \Omega$. For any $n \in \N$, $(u_n,v_n) \in \Omega$. Thus, $u_n \in \E$ and $v_n = \DM u_n$. Consequently, we have:
\begin{equation}
u_n \xrightarrow[]{\L^p} u \quad \text{and} \quad D^\alpha_- u_n \xrightarrow[]{\L^p} v.
\end{equation}
For any $n \in \N$, since $u_n \in \E$ and $I^\alpha_-$ is continuous from $\L^p$ to $\L^p$, we have:
\begin{equation}
u_n = I^\alpha_- \circ D^\alpha_- u_n \xrightarrow[]{\L^p} I^\alpha_- v.
\end{equation}
Thus, $u= I^\alpha_- v$, $D^\alpha_- u = D^\alpha_- \circ I^\alpha_- v = v \in \L^p$ and $I^\alpha_- \circ D^\alpha_- u = I^\alpha_- v = u$. Hence, $u \in \E$ and $(u,v) = (u,\DM u) \in \Omega$. In conclusion, $\Omega$ is a closed subspace of $\big((\L^p)^2, \Vert \cdot \Vert_{(\L^p)^2}\big)$ and then $\Omega$ is a reflexive separable Banach space. Finally, defining the following operator:
\begin{equation}
\fonction{A}{\E}{\Omega}{u}{(u,\DM u),}
\end{equation}
we prove that $\E$ is isometric isomorphic to $\Omega$. This completes the proof of Section~\ref{section32}.

\subsection{The continuous embedding $\E \hookrightarrow \CC_a$}\label{section33b}
Let us prove this result. Let $u \in \E$ and then $D^\alpha_- u \in \L^p$. Since $ 0 < (1/p) < \alpha < 1 $, Property~\ref{property4} leads to $I^\alpha_- \circ D^\alpha_- u \in \CC_a$. Furthermore, $u = I^\alpha_- \circ D^\alpha_- u$ and consequently, $u$ can be identified to its continuous representative. Finally, Property~\ref{property4} also gives:
\begin{equation}
\forall u \in \E, \; \Vert u \Vert_{\infty} = \Vert I^\alpha_- \circ D^\alpha_- u \Vert_{\infty} \leq  \dfrac{(b-a)^{\alpha -(1/p)}}{\Gamma (\alpha) \big( (\alpha -1) q + 1\big)^{1/q} } \vert u \vert_{\alpha,p}.
\end{equation}
Since $ \Vert \cdot \Vert_{\alpha,p}$ and $ \vert \cdot \vert_{\alpha,p}$ are equivalent norms, the proof of Section~\ref{section33b} is complete.

\subsection{The compact embedding $\E \hooktwoheadrightarrow \CC_a$}\label{section34}
Let us prove this property. Since $\E$ is a reflexive Banach space, we only have to prove that:
\begin{equation}\forall (u_n)_{n \in \N} \subset \E \; \text{such that} \; u_n \xrightharpoonup[]{\E} u, \; \text{then} \; u_n  \xrightarrow[]{\CC}  u.
\end{equation}
Let $(u_n)_{n \in \N} \subset \E$ such that:
\begin{equation}
u_n \xrightharpoonup[]{\E} u.
\end{equation}
Since $\E \hookrightarrow \CC_a$, we have:
\begin{equation}
u_n \xrightharpoonup[]{\CC} u.
\end{equation}
Since $(u_n)_{n \in \N}$ converges weakly in $\E$, $(u_n)_{n \in \N}$ is bounded in $\E$. Consequently, $(\DM u_n)_{n \in \N}$ is bounded in $\L^p$ by a constant $M \geq 0$. Let us prove that $(u_n)_{n \in \N} \subset \CC_a$ is uniformly lipschitzian on $[a,b]$. According to the proof of Property~\ref{property4}, we have:
\begin{eqnarray*}
\forall n \in \N, \; \forall a \leq t_1 < t_2 \leq b, \; \Vert u_n (t_2) - u_n(t_1) \Vert & \leq  &\Vert I^\alpha_- \circ \DM u_n (t_2) - I^\alpha_- \circ \DM u_n(t_1) \Vert \\
& \leq & \dfrac{2 \Vert \DM u_n \Vert_{\L^p}}{\Gamma (\alpha) \big( (\alpha -1) q +1 \big)^{1/p}} (t_2 -t_1)^{\alpha -(1/p)} \\
& \leq & \dfrac{2M}{\Gamma (\alpha) \big( (\alpha -1) q +1 \big)^{1/p}} (t_2 -t_1)^{\alpha -(1/p)} .
\end{eqnarray*}
Hence, from Ascoli's theorem, $(u_n)_{n \in \N}$ is relatively compact in $\CC$. Consequently, there exists a subsequence of $(u_n)_{n \in \N}$ converging strongly in $\CC$ and the limit is $u$ by uniqueness of the weak limit. \\

Now, let us prove by contradiction that the whole sequence $(u_{n})_{n \in \N}$ converges strongly to $u$ in $\CC$. If not, there exist $\varepsilon >0$ and a subsequence $(u_{n_k})_{k \in \N}$ such that:
\begin{equation}\label{eq31}
\forall k \in \N, \; \Vert u_{n_k} - u \Vert_{\infty} > \varepsilon > 0.
\end{equation}
Nevertheless, since $(u_{n_k})_{k \in \N}$ is a subsequence of $(u_{n})_{n \in \N}$, then it satisfies:
\begin{equation}
u_{n_k} \xrightharpoonup[]{\E} u.
\end{equation}
In the same way (using Ascoli's theorem), we can construct a subsequence of $(u_{n_k})_{k \in \N}$ converging strongly to $u$ in $\CC$ which is a contradiction to \eqref{eq31}. The proof of Section~\ref{section34} is now complete.

\subsection{Remarks}\label{section35}
Let us remind the following property:
\begin{equation}
\forall \varphi \in \CC^\infty_c, \; I^\alpha_- \varphi \in \CC^\infty_a.
\end{equation}
From this result, we get the two following results:
\begin{itemize}
\item $\CC^\infty_a$ is dense in $\E$. Indeed, let us first prove that $\CC^\infty_a \subset \E$. Let $u \in \CC^\infty_a \subset \L^p$. Since $u \in AC_a$ and $\dot{u} \in \L^p$, we have $\DM u = I^{1-\alpha}_- \dot{u} \in \L^p$. Since $u \in AC$, we also have $I^\alpha_- \circ \DM u = u$. Finally, $u \in \E$. Now, let us prove that $\CC^\infty_a$ is dense in $\E$. Let $u \in \E$, then $\DM u \in \L^p$. Consequently, there exists $(v_n)_{n \in \N} \subset \CC^\infty_c$ such that:
\begin{equation}
v_n \xrightarrow[]{\L^p} \DM u \quad \text{and then} \quad I^\alpha_- v_n \xrightarrow[]{\L^p} I^\alpha_- \circ \DM u = u,
\end{equation}
since $I^\alpha_-$ is continuous from $\L^p$ to $\L^p$. Defining $u_n := I^\alpha_- v_n \in \CC^\infty_a $ for any $n \in \N$, we obtain:
\begin{equation}
u_n \xrightarrow[]{\L^p} u \quad \text{and} \quad \DM u_n = \DM \circ I^\alpha_- v_n = v_n \xrightarrow[]{\L^p} \DM u.
\end{equation}
Finally, $(u_n)_{n \in \N} \subset \CC^\infty_a $ and converges to $u$ in $\E$. The proof of this point is complete; \\
\item In the case $(1/p) < \min (\alpha,1-\alpha)$, $\E = \{ u \in \L^p \; \text{satisfying} \; \DM u \in \L^p  \}$. Indeed, let $u \in \L^p$ satisfying $ D^\alpha_- u \in \L^p$ and let us prove that $I^\alpha_- \circ D^\alpha_- u =u $. Let $\varphi \in \CC^\infty_c \subset \L^1$. Since $\DM u \in \L^p$, Property~\ref{property3} leads to:
\begin{equation}
\di \int_a^b I^\alpha_- \circ \DM u \cdot \varphi \; dt = \di \int_a^b \DM u \cdot I^\alpha_+ \varphi \; dt = \di \int_a^b \dfrac{d}{dt} (I^{1-\alpha}_- u) \cdot I^\alpha_+ \varphi \; dt.
\end{equation}
Then, an integration by parts gives:
\begin{equation}
\di \int_a^b I^\alpha_- \circ \DM u \cdot \varphi \; dt = \di \int_a^b I^{1-\alpha}_- u \cdot D^{1-\alpha}_+ \varphi \; dt.
\end{equation}
Indeed, $I^{\alpha}_+ \varphi (b) = 0$ since $\varphi \in \CC^\infty_c$ and $I^{1-\alpha}_- u (a) = 0$ since $u \in \L^p$ and $ (1/p) < 1 - \alpha $. Finally, using Property~\ref{property3} again, we obtain:
\begin{equation}
\di \int_a^b I^\alpha_- \circ \DM u \cdot \varphi \; dt = \di \int_a^b  u \cdot I^{1-\alpha}_+ \circ D^{1-\alpha}_+ \varphi \; dt = \di \int_a^b  u \cdot \varphi \; dt,
\end{equation}
which concludes the proof of this second point. In this case, let us note that such a definition of $\E$ could lead us to name it \textit{fractional Sobolev space} and to denote it by $\W^{\alpha,p}$. Nevertheless, these notion and notation are already used, see \cite{brez}.
\end{itemize}

\section{Variational structure of \eqref{elf}}\label{section4}
In the rest of the paper, we assume that Lagrangian $L$ is of class $\CC^1$ and we define the Lagrangian functional $\LL$ on $\E$ (with $0<(1/p)<\alpha <1$). Precisely, we define:
\begin{equation}
\fonction{\LL}{\E}{\R}{u}{\di \int_a^b L(u,\DM u,t) \; dt.}
\end{equation}
$\LL$ is said to be G\^ateaux-differentiable in $u \in \E$ if the map:
\begin{equation}
\fonction{D\LL (u)}{\E}{\R}{v}{D\LL (u)(v) := \lim\limits_{h \to 0} \dfrac{\LL (u+hv)- \LL(u)}{h} }
\end{equation}
is well-defined for any $v \in \E$ and if it is linear and continuous. A critical point $u \in \E$ of $\LL$ is defined by $D\LL (u)=0$.

\subsection{G\^ateaux-differentiability of $\LL$}\label{section41}
Let us prove the following lemma:
\begin{lemma}\label{lem1}
The following implications hold:
\begin{itemize}
\item $L$ satisfies \eqref{h1} $\Longrightarrow$  for any $u \in \E$, $L(u,\DM u,t) \in \L^1$ and then $\LL (u)$ exists in $ \R$;
\item $L$ satisfies \eqref{h2} $\Longrightarrow$ for any $u \in \E$, $ \partial L / \partial x (u,\DM u,t) \in \L^1$;
\item $L$ satisfies \eqref{h3} $\Longrightarrow$ for any $u \in \E$, $ \partial L / \partial y (u,\DM u,t) \in \L^q$.
\end{itemize}
\end{lemma}

\begin{proof}
Let us assume that $L$ satisfies \eqref{h1} and let $u \in \E \subset \CC_a$. Then, $\Vert \DM u \Vert^{d_1} \in \L^{p/d_1} \subset \L^1$ and the three maps $t \longrightarrow r_1 \big(u(t),t\big)$, $s_1 \big(u(t),t\big)$, $\vert L\big(u(t),0,t\big) \vert \in \CC([a,b],\R^+) \subset \L^\infty \subset \L^1$. Hypothesis \eqref{h1} implies for almost all $t \in [a,b]$:
\begin{equation}
\vert L(u(t),\DM u (t),t) \vert \leq r_1 (u(t),t) \Vert \DM u(t) \Vert^{d_1}+s_1 (u(t),t) + \vert L(u(t),0,t) \vert.
\end{equation}
Hence, $L(u,\DM u,t) \in \L^1$ and then $\LL (u)$ exists in $\R$. We proceed in the same manner in order to prove the second point of Lemma~\ref{lem1}. Now, assuming that $L$ satisfies \eqref{h3}, we have $\Vert \DM u \Vert^{d_3} \in \L^{p/d_3} \subset \L^q$ for any $u \in \E$. An analogous argument gives the third point of Lemma~\ref{lem1}.
\end{proof}

Let us prove the following result:

\begin{proposition}\label{prop1}
Assuming that $L$ satisfies Hypotheses \eqref{h1}, \eqref{h2} and \eqref{h3}, $\LL$ is G\^ateaux-differentiable in any $u \in \E$ and:
\begin{equation}
\forall u, v \in \E, \; D\LL (u)(v) = \di \int_a^b \dfrac{\partial L}{\partial x} (u,\DM u,t) \cdot v + \dfrac{\partial L}{\partial y} (u,\DM u,t) \cdot \DM v \; dt.
\end{equation}
\end{proposition}

\begin{proof}
Let $u$, $v \in \E \subset \CC_a$. Let $\psi_{u,v}$ defined for any $h \in [-1,1]$ and for almost all $t \in [a,b]$ by:
\begin{equation}
\psi_{u,v} (t,h) : = L\big(u(t)+hv(t),\DM u(t) + h \DM v(t),t\big).
\end{equation}
Then, we define the following mapping:
\begin{equation}
\fonction{\phi_{u,v}}{[-1,1]}{\R}{h}{\di \int_a^b L(u+hv,\DM u + h \DM v,t) \; dt = \di \int_a^b \psi_{u,v} (t,h) \; dt.}
\end{equation}
Our aim is to prove that the following term:
\begin{equation}
D\LL (u)(v) = \lim\limits_{h \to 0} \dfrac{\LL (u+hv) - \LL (u)}{h} = \lim\limits_{h \to 0} \dfrac{\phi_{u,v} (h) - \phi_{u,v}(0)}{h} = \phi_{u,v}'(0)
\end{equation}
exists in $\R$. In order to differentiate $\phi_{u,v}$, we use the theorem of differentiation under the integral sign. Indeed, we have for almost all $t \in [a,b]$, $\psi_{u,v}(t,\cdot)$ is differentiable on $[-1,1]$ with:
\begin{multline}
\forall h \in [-1,1], \; \dfrac{\partial \psi_{u,v}}{\partial h} (t,h) = \dfrac{\partial L}{\partial x} \big( u(t)+hv(t),\DM u (t)+h \DM v(t),t \big) \cdot v (t) \\ + \dfrac{\partial L}{\partial y} \big( u(t)+hv(t),\DM u (t)+h \DM v(t),t \big) \cdot \DM v (t).
\end{multline}
Then, from Hypotheses \eqref{h2} and \eqref{h3}, we have for any $h \in [-1,1]$ and for almost all $t \in [a,b]$:
\begin{multline}
\left\vert \dfrac{\partial \psi_{u,v}}{\partial h} (t,h) \right\vert \leq \Big[ r_2 \big( u(t)+hv(t),t \big) \Vert \DM u(t) + h \DM v (t) \Vert^{d_2} + s_2 \big( u(t)+hv(t),t \big) \Big] \Vert v(t) \Vert \\ + \Big[ r_3 \big( u(t)+hv(t),t \big) \Vert \DM u(t) + h \DM v (t) \Vert^{d_3} + s_3 \big( u(t)+hv(t),t \big) \Big] \Vert \DM v(t)\Vert.
\end{multline}
We define:
\begin{equation}
r_{2,0} := \max\limits_{(t,h) \in [a,b] \times [-1,1]} r_2 \big( u(t)+hv(t),t \big) 
\end{equation}
and we define similarly $s_{2,0}$, $r_{3,0}$, $s_{3,0}$. Finally, it holds:
\begin{multline}
\left\vert \dfrac{\partial \psi_{u,v}}{\partial h} (t,h) \right\vert \leq 2^{d_2} r_{2,0} ( \underbrace{  \Vert \DM u(t) \Vert^{d_2} + \Vert \DM v (t) \Vert^{d_2} }_{\in \L^{p/d_2} \subset \L^1}) \underbrace{ \Vert v(t)  \Vert }_{\in \CC_a \subset \L^\infty} + s_{2,0} \underbrace{ \Vert v(t)  \Vert }_{\in \CC_a \subset \L^1} \\ + 2^{d_3} r_{3,0} \underbrace{  \Vert \DM u(t) \Vert^{d_3} + \Vert \DM v (t) \Vert^{d_3} }_{\in \L^{p/d_3} \subset \L^q}) \underbrace{ \Vert \DM v(t)  \Vert }_{\in \L^p} + s_{3,0} \underbrace{ \Vert \DM v(t)  \Vert }_{\in \L^p \subset \L^1}.
\end{multline}
The right term is then a $\L^1$ function independent of $h$. Consequently, applying the theorem of differentiation under the integral sign, $\phi_{u,v}$ is differentiable with:
\begin{eqnarray}
\forall h \in [-1,1], \; \phi_{u,v}'(h) = \di \int_a^b \dfrac{\partial \psi_{u,v}}{\partial h} (t,h) \; dt .
\end{eqnarray}
Hence:
\begin{equation}
D\LL (u)(v) = \phi_{u,v}'(0) = \di \int_a^b \dfrac{\partial \psi_{u,v}}{\partial h} (t,0) \; dt = \di \int_a^b \dfrac{\partial L}{\partial x} (u,\DM u,t) \cdot v + \dfrac{\partial L}{\partial y} (u,\DM u,t) \cdot \DM v \; dt.
\end{equation}
From Lemma~\ref{lem1}, it holds:
\begin{equation}
\dfrac{\partial L}{\partial x} (u,\DM u,t) \in \L^1 \; \text{and} \; \dfrac{\partial L}{\partial y} (u,\DM u,t) \in \L^q.
\end{equation}
Since $v \in \CC_a \subset \L^\infty$ and $\DM v \in \L^p$, $D\LL (u)(v)$ exists in $\R$. Moreover, we have:
\begin{eqnarray*}
\vert D\LL (u)(v) \vert &  \leq  & \left\Vert \dfrac{\partial L}{\partial x} (u,\DM u,t) \right\Vert_{\L^1} \Vert v \Vert_{\infty} + \left\Vert \dfrac{\partial L}{\partial y} (u,\DM u,t) \right\Vert_{\L^q} \Vert \DM v \Vert_{\L^p} \\
& \leq &  \left( \dfrac{(b-a)^{\alpha -(1/p)}}{\Gamma (\alpha) \big( (\alpha-1)q+1 \big)^{1/q}} \left\Vert \dfrac{\partial L}{\partial x} (u,\DM u,t) \right\Vert_{\L^1} + \left\Vert \dfrac{\partial L}{\partial y} (u,\DM u,t) \right\Vert_{\L^q} \right) \vert v \vert_{\alpha,p}.
\end{eqnarray*}
Consequently, $D\LL (u)$ is linear and continuous from $\E$ to $\R$. The proof is complete. 
\end{proof}

\subsection{Sufficient condition for a weak solution}\label{section42}
In this section, we prove the following theorem:

\begin{theorem}\label{thm1}
Let us assume that $L$ satisfies Hypotheses \eqref{h1}, \eqref{h2} and \eqref{h3}. Then:
\begin{equation}
u \; \text{ is a critical point of} \; \LL \Longrightarrow u \; \text{is a weak solution of \eqref{elf}}. 
\end{equation}
\end{theorem}

\begin{proof}
Let $u$ be a critical point of $\LL$. Then, we have in particular:
\begin{equation}
\forall v \in \CC^\infty_c, \; D\LL (u)(v) = \di \int_a^b \dfrac{\partial L}{\partial x} (u,\DM u,t) \cdot v + \dfrac{\partial L}{\partial y} (u,\DM u,t) \cdot \DM v \; dt = 0.
\end{equation}
For any $v \in \CC^\infty_c \subset AC_a$, $\DM v = I^{1-\alpha}_- \dot{v} \in \CC^\infty_a$. Since $\partial L / \partial y (u,\DM u,t) \in \L^q$, Property~\ref{property3} gives:
\begin{equation}
\forall v \in \CC^\infty_c, \; \di \int_a^b \dfrac{\partial L}{\partial x} (u,\DM u,t) \cdot v + I^{1-\alpha}_+ \left( \dfrac{\partial L}{\partial y} (u,\DM u,t) \right) \cdot \dot{v} \; dt = 0.
\end{equation}
Finally, we define:
\begin{equation}
\forall t \in [a,b], \; w_u(t) = \di \int_a^t \dfrac{\partial L}{\partial x} (u,\DM u,t) \; dt .
\end{equation}
Since $\partial L / \partial x (u,\DM u,t) \in \L^1$, $w_u \in AC_a$ and $\dot{w_u} = \partial L / \partial x (u,\DM u,t)$. Then, an integration by parts leads to:
\begin{equation}
\forall v \in \CC^\infty_c, \; \di \int_a^b \left( I^{1-\alpha}_+ \left( \dfrac{\partial L}{\partial y} (u,\DM u,t) \right) - w_u \right) \cdot \dot{v} \; dt = 0.
\end{equation}
Consequently, there exists a constant $C \in \R^d$ such that:
\begin{equation}
I^{1-\alpha}_+ \left( \dfrac{\partial L}{\partial y} (u,\DM u,t)\right)= C+ w_u \in AC.
\end{equation}
By differentiation, we obtain:
\begin{equation}
-\DP \left( \dfrac{\partial L}{\partial y} (u,\DM u,t) \right) = \dfrac{\partial L}{\partial x} (u,\DM u,t),
\end{equation}
and then $u \in \E \subset \CC$ satisfies \eqref{elf} a.e. on $[a,b]$. The proof is complete.
\end{proof}

Let us note that the use of Property~\ref{property3} in the previous proof leads to the emergence of $\DP$ in \eqref{elf} although $\LL$ is only dependent of $\DM$. This asymmetry in \eqref{elf} is a strong drawback in order to solve it explicitly. However, from Theorem~\ref{thmprincipal}, the existence of a weak solution for \eqref{elf} will be guarantee.

\section{Existence of a global minimizer of $\LL$}\label{section5}
In this section, under assumptions \eqref{h4} and \eqref{h5}, we prove the existence of a global minimizer $u$ of $\LL$, see Theorem~\ref{thm2}. Then, $u$ is a critical point of $\LL$ and then, according to Theorem~\ref{thm1}, $u$ is a weak solution of \eqref{elf}. This concludes the proof of Theorem~\ref{thmprincipal}. \\

As usual in a variational method, in order to prove the existence of a global minimizer of a functional, coercivity and convexity hypotheses need to be added on the Lagrangian. We have already define Hypotheses \eqref{h4} (coercivity) and \eqref{h5} (convexity) in Section~\ref{section12}. In this section, we introduce two different convexity hypotheses \eqref{h5p} and \eqref{h5pp} under which Theorem~\ref{thmprincipal} is still valid:
\begin{itemize}
\item Convexity hypothesis denoted by \eqref{h5p}: 
\begin{multline}\tag{H${}^\prime_5$}\label{h5p}
\forall (x,t) \in \R^d \times [a,b], \; L(x,\cdot ,t) \; \text{is convex} \\ \text{and} \; \big( L(\cdot ,y,t) \big)_{(y,t) \in \R^d \times [a,b]} \; \text{is uniformly equicontinuous on} \; \R^d .
\end{multline}
We remind that the uniform equicontinuity of $\big( L(\cdot,y,t) \big)_{(y,t) \in \R^d \times [a,b]}$ has to be understood as:
\begin{multline}
\forall \varepsilon > 0, \; \exists \delta >0, \; \forall (x_1 , x_2) \in (\R^d)^2, \\ \Vert x_2 - x_1 \Vert < \delta \Longrightarrow \forall (y,t) \in \R^d \times [a,b], \vert L(x_2,y,t) - L(x_1,y,t) \vert < \varepsilon.
\end{multline}
Let us note that Hypotheses \eqref{h5} and \eqref{h5p} are independent. \\
\item Convexity hypothesis denoted by \eqref{h5pp}:
\begin{equation}\tag{H${}^{\prime \prime}_5$}\label{h5pp}
\forall (x,t) \in \R^d \times [a,b], \; L(x,\cdot ,t) \; \text{is convex}.
\end{equation}
\end{itemize}
Hypothesis \eqref{h5pp} is the weakest. Nevertheless, in this case, the detailed proof of Theorem~\ref{thm2} is more complicated. Consequently, in the case of Hypothesis \eqref{h5pp}, we do not develop the proof and we use a strong result proved in \cite{daco2}. \\

Let us prove the following preliminary result:

\begin{lemma}\label{lem2}
Let us assume that $L$ satisfies Hypothesis \eqref{h4}. Then, $\LL$ is coercive in the sense that:
\begin{equation}
\lim\limits_{\Vert u \Vert_{\alpha,p} \to +\infty} \LL (u) = +\infty.
\end{equation}
\end{lemma}

\begin{proof}
Let $u \in \E$, we have:
\begin{equation}
\LL (u) = \di \int_a^b L(u,\DM u,t) \; dt \geq \di \int_a^b c_1(u,t) \Vert \DM u \Vert^p + c_2 (t) \Vert u \Vert^{d_4}+c_3(t) \; dt.
\end{equation}
Equation \eqref{eq54} implies that: 
\begin{equation}
\Vert u \Vert_{\L^{d_4}}^{d_4} \leq (b-a)^{1-\frac{d_4}{p}} \Vert u \Vert_{\L^{p}}^{d_4} \leq \dfrac{(b-a)^{\alpha +1-\frac{d_4}{p}}}{\Gamma (\alpha +1)} \Vert \DM u \Vert_{\L^{p}}^{d_4} =  \dfrac{(b-a)^{\alpha +1-\frac{d_4}{p}}}{\Gamma (\alpha +1)} \vert  u \vert_{\alpha,p}^{d_4}.
\end{equation}
Finally, we conclude that:
\begin{eqnarray}
\forall u \in \E, \; \LL (u) & \geq & \gamma \Vert \DM u \Vert^p_{\L^p} - \Vert c_2 \Vert_{\infty} \Vert u \Vert_{\L^{d_4}}^{d_4} - (b-a) \Vert c_3 \Vert_{\infty} \\
& \geq & \gamma \vert u \vert^p_{\alpha,p} - \dfrac{\Vert c_2 \Vert_{\infty} (b-a)^{\alpha +1-\frac{d_4}{p}}}{\Gamma (\alpha +1)} \vert u \vert_{\alpha,p}^{d_4} - (b-a) \Vert c_3 \Vert_{\infty} .
\end{eqnarray}
Since $d_4 < p$ and since the norms $\vert \cdot \vert_{\alpha,p}$ and $\Vert \cdot \Vert_{\alpha,p}$ are equivalent, the proof is complete.
\end{proof}
 
Now, we are ready to prove Theorem~\ref{thm2}:

\begin{theorem}\label{thm2}
Let us assume that $L$ satisfies Hypotheses \eqref{h1}, \eqref{h2}, \eqref{h3}, \eqref{h4} and one of Hypotheses \eqref{h5}, \eqref{h5p} or \eqref{h5pp}. Then, $\LL$ admits a global minimizer.
\end{theorem}

\begin{proof}
Let $(u_n)_{n \in \N}$ be a sequence in $\E$ satisfying:
\begin{equation}
\LL (u_n) \longrightarrow \inf\limits_{v \in \E} \LL (v) =: K.
\end{equation}
Since $L$ satisfies Hypothesis \eqref{h1}, $\LL (u) \in \R$ for any $u \in \E$. Hence, $K < +\infty$. Let us prove by contradiction that $(u_n)_{n \in \N}$ is bounded in $\E$. In the negative case, we can construct a subsequence $(u_{n_k})_{k \in \N}$ satisfying $\Vert u_{n_k} \Vert_{\alpha,p} \to + \infty$. Since $L$ satisfies Hypothesis \eqref{h4}, Lemma~\ref{lem2} gives:
\begin{equation}
K = \lim\limits_{k \in \N} \LL (u_{n_k}) = +\infty,
\end{equation}
which is a contradiction. Hence, $(u_n)_{n \in \N}$ is bounded in $\E$. Since $\E$ is reflexive, there exists a subsequence still denoted by $(u_n)_{n \in \N}$ converging weakly in $\E$ to an element denoted by $u \in \E$. Let us prove that $u$ is a global minimizer of $\LL$. Since:
\begin{equation}
u_n  \xrightharpoonup[]{\E} u \quad \text{and} \quad \E \hooktwoheadrightarrow \CC_a, 
\end{equation}
we have:
\begin{equation}\label{eq5}
u_n \xrightarrow[]{\CC} u \quad \text{and} \quad \DM u_n \xrightharpoonup[]{\L^p} \DM u.
\end{equation}
\textit{Case $L$ satisfies \eqref{h5}:} by convexity, it holds for any $n \in \N$:
\begin{multline}
\LL (u_n) =  \di \int_a^b L(u_n, \DM u_n ,t) \; dt \geq  \di \int_a^b L(u, \DM u ,t) \; dt \\ + \di \int_a^b \dfrac{\partial L}{\partial x}(u, \DM u ,t) \cdot (u_n -u) \; dt + \di \int_a^b \dfrac{\partial L}{\partial y}(u, \DM u ,t) \cdot (\DM u_n - \DM u) \; dt.
\end{multline}
Since $L$ satisfies Hypotheses \eqref{h2} and \eqref{h3}, $\partial L / \partial x(u, \DM u ,t) \in \L^1$ and $\partial L / \partial y(u, \DM u ,t) \in \L^q$. Consequently, using \eqref{eq5} and making $n$ tend to $+\infty$, we obtain:
\begin{equation}
K = \inf\limits_{v \in \E} \LL (v) \geq \di \int_a^b L(u, \DM u ,t) \; dt = \LL (u).
\end{equation}
Consequently, $u$ is a global minimizer of $\LL$. \\

\textit{Case $L$ satisfies \eqref{h5p}:} let $\varepsilon >0$. Since $(u_n)_{n \in \N}$ converges strongly in $\CC$ to $u$, we have:
\begin{equation}
\exists N \in \N, \; \forall n \geq N, \; \Vert u_n -u \Vert_{\infty} < \delta,
\end{equation}
where $ \delta$ is given in the definition of \eqref{h5p}. In consequence, it holds a.e. on $[a,b]$:
\begin{equation}\label{eq6}
\forall n \geq N, \;  \vert L \big( u_n(t), \DM u_n (t),t\big) - L \big( u(t), \DM u_n (t),t\big) \vert < \varepsilon.
\end{equation}
Moreover, for any $n \geq N$, we have:
\begin{multline}
\LL (u_n) =  \di \int_a^b L(u,\DM u,t) \; dt  + \di \int_a^b L(u_n,\DM u_n,t) - L(u,\DM u_n,t) \; dt \\ + \di \int_a^b L(u,\DM u_n,t) -  L(u,\DM u,t) \; dt.
\end{multline}
Then, for any $n \geq N$, it holds by convexity:
\begin{multline}
\LL (u_n) \geq  \di \int_a^b L(u,\DM u,t) \; dt  - \di \int_a^b \vert L(u_n,\DM u_n,t) - L(u,\DM u_n,t) \vert \; dt \\ + \di \int_a^b \dfrac{\partial L}{\partial y}(u, \DM u, t) \cdot (\DM u_n - \DM u) \; dt.
\end{multline}
And, using Equation \eqref{eq6}, we obtain for any $n \geq N$:
\begin{equation}\label{eq7}
\LL (u_n) \geq  \di \int_a^b L(u,\DM u,t) \; dt  - \varepsilon (b-a)  + \di \int_a^b \dfrac{\partial L}{\partial y}(u, \DM u, t) \cdot (\DM u_n - \DM u) \; dt.
\end{equation}
We remind that $\partial L / \partial y(u, \DM u ,t) \in \L^q$ since $L$ satisfies \eqref{h3}. Since  $(\DM u_n)_{n \in \N}$ converges weakly in $\L^p$ to $\DM u$, we obtain by making $n$ tend to $+\infty$ and then by making $\varepsilon$ tend to $0$:
\begin{equation}
K = \inf\limits_{v \in \E} \LL (v) \geq \di \int_a^b L(u, \DM u ,t) \; dt = \LL (u).
\end{equation}
Consequently, $u$ is a global minimizer of $\LL$. \\

\textit{Case $L$ satisfies \eqref{h5pp}:} we refer to Theorem 3.23 in \cite{daco2}.
\end{proof}

Finally, combining Theorems~\ref{thm1} and~\ref{thm2}, the proof of Theorem~\ref{thmprincipal} is now complete.

\section{Examples}\label{section6}
Let us consider some examples of Lagrangian $L$ satisfying Hypotheses of Theorem~\ref{thmprincipal}. Consequently, the fractional Euler-Lagrange equation \eqref{elf} associated admits a weak solution $u \in \E$. \\

The most classical example is the Dirichlet integral, \textit{i.e.} the Lagrangian functional associated to the Lagrangian $L$ given by:
\begin{equation}
L(x,y,t) = \dfrac{1}{2} \Vert y \Vert^2.
\end{equation}
In this case, $L$ satisfies Hypotheses \eqref{h1}, \eqref{h2}, \eqref{h3}, \eqref{h4} and \eqref{h5} for $p = 2$. Hence, the fractional Euler-Lagrange equation \eqref{elf} associated admits a weak solution in $\E$ for $(1 /2) < \alpha < 1$. \\

In a more general case, the following Lagrangian $L$:
\begin{equation}
L(x,y,t) = \dfrac{1}{p} \Vert y \Vert^p + a(x,t),
\end{equation}
where $p > 1$ and $a \in \CC^1(\R^d \times [a,b],\R^+)$, satisfies Hypotheses \eqref{h1}, \eqref{h2}, \eqref{h3}, \eqref{h4} and \eqref{h5pp}. Consequently, the fractional Euler-Lagrange equation \eqref{elf} associated to $L$ admits a weak solution in $\E$ for any $(1 /p) < \alpha < 1$. Let us note that if for any $t \in [a,b]$, $a(\cdot,t)$ is convex, then $L$ satisfies Hypothesis \eqref{h5}. \\

In the unidimensional case $d=1$, let us take a Lagrangian with a second term linear in its first variable, \textit{i.e.}:
\begin{equation}
L(x,y,t) = \dfrac{1}{p} \vert y \vert^p + f(t) x,
\end{equation}
where $p > 1$ and $f \in \CC^1([a,b],\R)$. Then, $L$ satisfies Hypotheses \eqref{h1}, \eqref{h2}, \eqref{h3}, \eqref{h4} and \eqref{h5}. Then, the fractional Euler-Lagrange equation \eqref{elf} associated admits a weak solution in $\E$ for any $(1 /p) < \alpha < 1$. \\

Theorem~\ref{thmprincipal} is a result based on strong conditions on Lagrangian $L$. Consequently, some Lagrangian do not satisfy all hypotheses of Theorem~\ref{thmprincipal}. We can cite the Bolza's example in dimension $d=1$ given by:
\begin{equation}
L(x,y,t) = (y^2 - 1)^2 + x^4.
\end{equation}
$L$ does not satisfy Hypothesis \eqref{h4} neither Hypothesis \eqref{h5pp}. Nevertheless, as usual with variational methods, the conditions of regularity, coercivity and/or convexity can often be replaced by weaker assumptions specific to the studied problem. As an example, we can cite \cite{torr6} and references therein about higher-order integrals of the calculus of variations. Indeed, in this paper, it is proved that calculus of variations is still valid with weaker regularity assumptions.

\section*{Conclusion}
The method developed in this paper gives a framework in order to study the existence of weak solutions for fractional Euler-Lagrange equations. In this paper, we have studied the special case of a Lagrangian functional involving fractional derivatives of Riemann-Liouville. Nevertheless, such a method can also be developed in the case of fractional derivatives of Caputo or Hadamard. Indeed, these operators satisfy similar properties than Riemann-Liouville's ones, see \cite{kilb,samk}. In fact, the same method can be developed in the case of general linear operators used in \cite{agra5,kiry,odzi,odzi2}: this is the aim of a forthcoming paper. 

\bibliographystyle{plain}

\end{document}